\documentclass{amsart}

\usepackage{amsmath,amsthm,amsfonts,amssymb,amscd}
\usepackage[all]{xy}
\usepackage[english]{babel}

\usepackage[dvips]{graphicx}
\usepackage[T1]{fontenc}
\usepackage[latin1]{inputenc}
\usepackage{ae}

{\theoremstyle{plain}
\newtheorem{theorem}{Theorem}[section]
\newtheorem{proposition}{Proposition}[section]
\newtheorem{lemma}{Lemma}[section]
\newtheorem{example}{Example}[section]
\newtheorem{definition}{Definition}[section]
\newtheorem{remark}{Remark}[section]
\newtheorem{corollary}{Corollary}[section]
}
\newcommand{\nc}{\newcommand}
\nc {\hh}{\check{h}}
\nc {\DD}{\mathcal{D}}
\nc {\CC}{\mathbb{C}}
\nc {\Pp}{\mathbb{P}}
\nc {\Ss}{\mathcal{S}}
\nc {\PP}{\mathbb{P}^{2}}
\nc {\Pd}{ \check{\mathbb{P}}^{2}}
\nc {\WW}{\mathcal{W}}
\nc {\Sym}{\mathrm{Sym}}
\nc {\OO}{\mathcal{O}}
\nc {\UU}{\mathcal{U}}
\nc {\EE}{\mathcal{E}}
\nc {\MM}{\mathcal{M}}
\nc {\KK}{\mathcal{K}}
\nc {\PW}{\mathcal{P}}
\nc {\NW}{\mathcal{N}_{\WW}}
\nc {\FF}{\mathcal{F}}
\nc {\GG}{\mathcal{G}}
\nc {\ZZ}{\mathcal{Z}}
\nc {\LL}{\mathcal{L}}
\nc {\NN}{\mathcal{N}}
\nc {\VV}{\mathcal{V}}
\nc {\Ww}{\mathbb{W}}
\nc {\QQ}{\mathbb{Q}}
\nc {\II}{\mathcal{I}}

\date{}

\begin{document}

\title[Second order differential equations without algebraic solutions]{On the density of Second Order Differential Equations without algebraic solutions on $\PP$}

\author{M. Falla Luza}


\begin{abstract}
We prove that a generic second order differential equation in the projective plane has no algebraic solutions when the bidegree is big enough. We also proof an analogous result for webs on $\PP$.
\end{abstract}
\maketitle 


\section{Introduction}

In the late 1970s, J. P. Jouanolou reworked and extended the work of Darboux
about differential equations over the complex projective plane \cite{Dar} in the framework of modern algebraic geometry. One of the most important results of Jouanolou's monograph \cite{Jou} states that a very generic holomorphic foliation of the projective plane, of degree at least 2, does not have any invariant algebraic curves.\\
 
This result was extended in various ways, see \cite{CouPer}, \cite{Alc}, \cite{Per} and \cite{Soa}. For example, in \cite{CouPer} the authors prove that over a smooth complex projective variety of dimension greater or equal than 2, a very generic holomorphic foliation of dimension one with sufficiently ample cotangent bundle has no proper invariant algebraic subvarieties of nonzero dimension. On the other hand, in \cite{Per} the author gives a different proof of Jouanolou's theorem following the ideas of \cite{CouPer} and restricting to $\PP$.\\

Since we can think in holomorphic foliations as first order differential equations of degree one, we are tempted to believe that the same assertion holds true for first order differential equations of any degree (webs) or for higher order differential equations.\\

Our main theorem shows that the assertion is true for second order differential equations on $\PP$, that is, we prove that a generic second order differential equation has no algebraic solutions for certain conditions over the bidegree (which we will define on Section 3). Our result is sharp in the sense that the second order differential equations not satisfying the above mentioned conditions always have invariant algebraic solutions.\\
 
We also proof that an analogous to Jouanolou's theorem holds true in the case of $k$-webs of degree d (first order differential equations of degree $k$) on $\PP$ when $d\geq 2$; and for webs with sufficiently ample normal bundle on projective surfaces.\\

The main idea is to show that the set of second order differential equations \-ha\-ving algebraic solution is a union of countable many closed sets in the set of the second order differential equations with fixed bidegree. Then we use the result for webs to show that all of this closed set are proper subsets.\\

The paper is organized as follows: In Section 2 we define our main object, second order differential equations on a surface and we describe the space we work with. We also give some algebraic formulas and use them to shed some light on the geometry of this equations. In Section 3 we introduce the space of second order differential equations and we describe this space for some cases. We finish this section stating our main result, Theorem \ref{jouanolou-para-ecuaciones-diferenciales}. We devote Section 4 to show a similar result for webs and give a known description of webs of low degree. In Section 5 we prove the main theorem. Finally in Section 6 we give a similar result for second order differential equations on general surfaces.

\section{Second Order Differential Equations}
\subsection{The contact distribution}
Let $S$ be a complex (smooth) surface and consider $M=\Pp(TS)$ the total space of the projective bundle $\pi: \Pp(TS)\rightarrow S$. The three dimensional variety $M$ is usually called the \textit{the contact variety}. For any curve $C \subseteq M$ one can define its lifting to $M$ as the curve: $\widetilde{C}=\overline{\{(x,[T_{x}C]): x\in C_{smooth}\}}\subseteq M$. \\

For each point $x=(z,[v])\in M$,i.e. $z\in S$ and $v\in T_{z}S$, one has the plane $\DD_{x}:=(d\pi(x))^{-1}(\CC v)$. We obtain in this way a two dimensional distribution $\DD$ in $M$, the so called \textit{contact distribution}:
\begin{equation}\label{secuencia-de-distribucion-de-contacto}
\xymatrix{
0 \ar[r] &\DD \ar[r]& TM \ar[r] & N_{\DD} \ar[r] & 0
}
\end{equation}

In local coordinates $(x,y,p)\in M$ (here (x,y) are coordinates in $S$ and $p$ represents the tangent direction $\frac{dy}{dx}$), $\DD$ is the distribution given by the 1-form $\alpha = dy-pdx$, called \textit{the contact form},
\[
\alpha \in H^{0}(M, \Omega^{1}_{M}\otimes N_{\DD}).
\]
The following lemma is clear by the definition and local form of $\DD$.

\begin{lemma}
The contact distribution satisfies the following properties:
\begin{enumerate}
    \item $\DD$ is not integrable.
    \item For any curve $C\subseteq S$ we have that $\widetilde{C}$ is tangent to $\DD$. Moreover, beside the fibers of $\pi$, the only curves on $M$ tangent to $\DD$ are the ones of this form.  
\end{enumerate}
\end{lemma}
\subsection{Second Order Differential Equations}
\begin{definition}
A second order differential equation in $S$ is a one dimensional foliation $\FF$ in $M$, tangent to $\DD$; in other words, a foliation defined by an element
$$
X_{\FF}\in H^{0}(M, \DD \otimes T^{*}\FF)
$$
for a suitable line bundle $T^{*}\FF$ (called the cotangent bundle of $\FF$). The solutions of this equations are the projections by $\pi$ of the leaves of $\FF_{sat}$, where $\FF_{sat}$ is the saturated foliation associated to $\FF$.
\end{definition}

\begin{remark}
For a clasical second order differential equation on $(\CC^{2},0)$ of the form

\begin{equation}\label{Second-order-dif-eq}
y''=\frac{A(x,y,y')}{B(x,y,y')}=\frac{a_{0}(x,y)+a_{1}(x,y)y'+\ldots +a_{l}(x,y)(y')^{l}}{b_{0}(x,y)+b_{1}(x,y)y'+\ldots +b_{k}(x,y)(y')^{k}}
\end{equation}
where the coefficients $a_{0}(x,y), \ldots , a_{l}(x,y), b_{0}(x,y), \ldots , b_{k}(x,y)$ are polynomials, one can associate the vector field 

\begin{equation*}
X=B(x,y,p)\frac{\partial}{\partial x}+pB(x,y,p)\frac{\partial}{\partial y}+A(x,y,p)\frac{\partial}{\partial p}
\end{equation*}
on $\Pp (T(\CC^{2},0))$. It is easy to verify that the integral curves of $X$ are projected by $\pi$ to the solutions of (\ref{Second-order-dif-eq}).\\

Observe also that $X$ is tangent to $\DD$; moreover every vector field on $\Pp (T(\CC^{2},0))$ tangent to $\DD$ has this form. Of course the foliation $\FF$ defined by $X$ extends to a foliation on $M=\Pp (T\PP)$ tangent to $\DD$. 
\end{remark}

We shall work in the case when $S=\PP$. Observe that in this case the variety $S$ can be identify with the incidence variety of points and lines in $\PP$. We denote by $\check{\pi}$ the restriction to $M$ of the projection over $\Pd$:
\[
\xymatrix{
&M=\Pp(T\PP) \ar[dl]_{\pi} \ar[dr]^{\check{\pi}} &\subseteq \PP \times \Pd \\
\PP& &\Pd
}
\]

\subsection{The cohomology of M}

Let us denote by $h=c_{1}(\OO_{\PP}(1))$ and $\hh=c_{1}(\OO_{\Pd}(1))$ the hyperplane classes on $\PP$ and $\Pd$ respectively. We still denote by $h$ and $\hh$ the respective pullbacks to $M$ by $\pi$ and $\check{\pi}$. Note that in coordinates $(x,y,p)$, $\hh$ is the class of the divisor $\{p=0\}$. Let us also denote by $\OO_{M}(a,b)$ the class of the line bundle $\OO_{M}(ah+b\hh)$. 

\begin{lemma}\label{normal-bundle-of-the-contac-distribution}
The class of the normal bundle of $\DD$ is given by $$N_{\DD}=\OO_{M}(1,1).$$
\end{lemma}

\begin{proof}
It is enough to write the contact form in all the coordinate charts of $M$ and observe that $N_{\DD}=\OO_{M}((\alpha)_{\infty}-(\alpha)_{0})$.
\end{proof}

On the other hand, looking to the sequences

\begin{equation}\label{diagrama-del-fibrado-normal-de-la-distribucion-de-contacto}
\xymatrix{&& 0 \ar[d] & 0 \ar[d] \\
0 \ar[r]&T_{M|S} \ar[r] \ar@{=}[d]&\DD \ar[d] \ar[r]&\OO_{\Pp(TS)}(-1) \ar[d] \ar[r]&0\\
0 \ar[r] & T_{M|S} \ar[r] &  TM \ar[d] \ar[r]^{d \pi} & \pi^{*}(TS) \ar[r] \ar[d] \ar[r] &0 \\
&&N_{\DD} \ar[d] &  T_{M|S}\otimes \OO_{\Pp(TS)}(-1) \ar[d] \\
&& 0 &0 }
\end{equation}
where the rigth vertical sequence is the Euler's sequence for projective bundles; we conclude that $N_{\DD}=\OO_{M}(3,0)\otimes \OO_{\Pp(TS)}(1)$. Therefore one obtains:

\begin{lemma}\label{tautological-bundle}
The class of the tautological bundle is given by $$c_{1}(\OO_{\Pp(TS)}(-1))=2h-\hh.$$
\end{lemma}

Note that for any complex vector bundle $\pi: E \rightarrow S$, the cohomology ring $H^{*}(\Pp(E))$ is, via the pullback map $\pi^{*}:H^{*}(S)\rightarrow H^{*}(\Pp(E))$ an algebra over the ring $H^{*}(S)$, wich is generated by $\xi = c_{1}(\mathcal{O}_{\Pp(E)}(-1))$ with the relation $\xi^{r} - c_{1}(E)\xi^{r-1} + c_{2}(E)\xi^{r-2} + \ldots + (-1)^{r-1}c_{r-1}(E)\xi + (-1)^{r}c_{r}(E)=0$  (see \cite{GH}, pg. 606).

Applying this to the case $E=T\PP$ and using lemma (\ref{tautological-bundle}) we get the desired description of $H^{*}(M)$:

\begin{equation*} \label{cohomologia-de-M}
H^{*}(M)=\frac{\mathbb{Z}[h,\hh]}{\langle h^{3}, h^{2} - h \hh +\hh^{2}\rangle} .
\end{equation*}

One also has:

\begin{lemma}\label{fibrado-canonico-de-M-fibrado-normal-y-determinante-de-D}
The following equalities hold true:
\begin{enumerate}
\item $K_{M}=\OO_{M}(-2,-2)$, where $K_{M}$ is the canonical bundle of $M$.
\item $N_{\DD}=\OO_{M}(1,1)=det(\DD)$.
\item $\hh ^{3}=0$, $h^{2}\hh=h \hh ^{2}=1$.
\end{enumerate}
\end{lemma}
\begin{proof}
The first asertion is a consequence of the diagram (\ref{diagrama-del-fibrado-normal-de-la-distribucion-de-contacto}) and lemma (\ref{tautological-bundle}). The second one follows from the sequence (\ref{secuencia-de-distribucion-de-contacto}) and lemma (\ref{normal-bundle-of-the-contac-distribution}). For the last equality observe that $h^{2}$ is the class of a fiber of $\pi$, then $h^{2}\hh=1$.
\end{proof}

\begin{example}\label{equaçao-das-retas}
Let us consider the differential equation satisfied by the lines: 
$$
y''=0
$$
We usually call $\LL$ the associated foliation in $M$. This foliation is given in coordinates $(x,y;p)$ by the vector field: $X_{\LL}=\frac{\partial}{\partial x} + p \frac{\partial}{\partial y}.$\\

Note that this foliation is tangent to the fibers of the projection $\check{\pi}:M\rightarrow \Pd$. Since $T^{*}\LL=\OO_{M}\left((X_{\LL})_{\infty}-(X_{\LL})_{0}\right)$, looking to $X_{\LL}$ in the other charts of M, we conclude
$$
T^{*}\LL=\OO_{M}(-2,1)
$$
\end{example}

\begin{example}\label{foliacion-vertical}
Let $\VV$ be the foliation tangent to the fibers of the projection $\pi:M\rightarrow \PP$. This foliation is given in the chart $(x,y;p)$ by the vectoy field $X_{\VV}=\frac{\partial}{\partial p}$, so $\VV$ is clearly tangent to $\DD$. In this case one has
$$
T^{*}\VV=\OO_{M}(1,-2).
$$
\end{example}
Observe that the follow sequence is exact:
\begin{equation}\label{distribucion-de-contacto-rectas-y-vertical}
\xymatrix{
0 \ar[r]& T\LL \ar[r]&\DD \ar[r]& T\VV \ar[r] & 0
}
\end{equation}

\subsection{Some intersection formulas}
In order to understand the geometric \-mea\-ning of $T\FF$ we give here some formulas concerning its intersection with curves and surfaces. Observe that for a second order differential equation $\FF$ on $\PP$, we can always write $T^{*}\FF=\OO_{M}(a,b)$ for some integers $a$ ,$b$. In this case we say that $\FF$ has bidegree $(a,b)$.\\  

Let $T\subseteq M$ be a compact surface, possibly singular, such that each irreducible component of $T$ is not $\FF$-invariant. We define the \textit{tangency curve} between $\FF$ and $T$ as the divisor on $T$ given locally by 
$$
tang(\FF,T)=\{X(F)|_{T}=0\}
$$ 
where $\{F=0\}$ is a local equation of $T$ and $X$ is a local holomorphic vector field generating $\FF$. Observe that this divisor can be defined for any complex compact subvariety $T\subseteq M$ of codimension 1 and any foliation by curves $\FF$ on a complex manifold $M$ (we do not need here the condition of tangency with $\DD$). Following the same proof of \cite{Brunella}, proposition 2 on page 23, one has:

\begin{proposition}\label{tangencia-entre-ecuacion-diferencial-y-superficie}
The tangency divisor between $\FF$ and $T$ is given by
$$
tang(\FF,T)=T^{*}\FF|_{T} + N_{T}.
$$
\end{proposition}

Let us consider now $\widetilde{C} \subseteq M$ a smooth compact curve, tangent to $\DD$, such that each component of $\widetilde{C}$ is not $\FF$-invariant and supose that the codimension of the singular set of $\FF$ is at least 2. Then we have the following diagram:

\[
\xymatrix{
&& 0 \ar[d]\\
&& T\FF|_{\widetilde{C}}\ar[d] \ar^{\sigma}[dr]\\
0 \ar[r]&  T\widetilde{C}\ar[r]&\DD|_{\widetilde{C}} \ar[r] &  \mathcal{N}_{\widetilde{C}}\ar[r]& 0
}
\]
where $\mathcal{N}_{\widetilde{C}}$ is the normal bundle of $\widetilde{C}$ in $\DD$. Observe that the map $\sigma$ vanishes exactly at the points where $\FF$ is tangent to $\widetilde{C}$. We define the tangency index between $\FF$ and $\widetilde{C}$ at a point $x \in \widetilde{C}$ as the vanishing order of the section induced by $\sigma$, $\OO_{\widetilde{C}}\rightarrow T^{*} \FF |_{\widetilde{C}}\otimes \mathcal{N}_{\widetilde{C}}$ (which we still denote by $\sigma$) at $x$
$$
tang(\FF,\widetilde{C},x)=ord_{x}(\sigma).
$$
Hence we can set 
$$
tang(\FF, \widetilde{C})=\sum_{x \in \widetilde{C}}tang(\FF, \widetilde{C}, x). $$
If one writes locally $\FF$ induced by $X=B\frac{\partial}{\partial x}+pB\frac{\partial}{\partial y}+A\frac{\partial}{\partial p}
$ and $\widetilde{C}$ parametrized by $(\gamma_{1},\gamma_{2},\gamma_{3})$, where $\gamma_{1}' \gamma_{3}=\gamma_{2}'$, then the points of tangency are exactly the points where $B(\gamma_{1},\gamma_{2},\gamma_{3})\gamma_{3}'-A(\gamma_{1},\gamma_{2},\gamma_{3})\gamma_{1}'$ vanishes and the tangency index at this points is the order of vanishing of this function; in particular $tang(\FF, \widetilde{C},x)\geq 0$. On the other hand, taking Chern classes to the previous sequence, one conclude the following proposition.

\begin{proposition}\label{tangencia-entre-ecuacion-diferencial-y-curva-no-invariante}
The tangency index between $\FF$ and a smooth compact currve $\widetilde{C}$ tangent to $\DD$ is
$$
tang(\FF, \widetilde{C})= T^{*}\FF. \widetilde{C} + det(\DD).\widetilde{C} - \chi(\widetilde{C})
$$
where $\chi(\widetilde{C})$ is the Euler characteristic of $\widetilde{C}$.
\end{proposition}

Now we can give a geometric interpretation of the bidegree of a second order differential equation.

\begin{corollary}\label{el-bigrado-de-una-ecuacion-diferencial}
Let $\FF$ be a saturated second order differential equation of bidegree $(a,b)$, which is neither $\LL$ (lines) nor $\VV$ (fibers) and denote by $F$ and $\widetilde{l}$ a fiber of $\pi$ and the lifting of a line in $\PP$ respectively, then:
\begin{eqnarray*}
a=tang(\FF, \widetilde{l})+1,\\
b=tang(\FF, F)+1.
\end{eqnarray*}
In particular one obtains: $a \geq 1$ and $b \geq 1$. 
\end{corollary}

As an example, we obtain the class of the cotangent bundle of the second order differential equation $\LL$ of the example (\ref{equaçao-das-retas}). Let us write $T^{*}\LL= \OO_{M}(a,b)$ and apply proposition (\ref{tangencia-entre-ecuacion-diferencial-y-curva-no-invariante}) for $\widetilde{C}=F$, a fiber of $\pi$, to obtain
$$
0=tang(\LL,F)=(ah+b\hh).h^{2}+(h+\hh).h^{2}-\chi(F)=b-1.
$$
On the other hand, observing that the tangency curve between $\LL$ and $H=\pi^{-1}(l)$, where $l$ is a line on $\PP$, is exactly $\widetilde{l}$, the lifting of $l$, we can apply proposition ($\ref{tangencia-entre-ecuacion-diferencial-y-superficie}$) to obtain
$$
\hh ^{2}=[tang(\LL,H)]=(ah+\hh).h+h^{2}=(a+2)h^{2}+\hh^{2}.
$$
We have just proved that $T^{*}\LL=\OO_{M}(-2,1)$. Of course, we can do a similar argument to obtain $T^{*}\VV=\OO_{M}(1,-2)$.\\

Take now $\FF_{1}$ and $\FF_{2}$ two second order differential equations on $S$ (not necessarily $\PP$). We define the \textit{tangency divisor} between $\FF_{1}$ and $\FF_{2}$ as the divisor on $M$ given locally by 
$$
tang(\FF_{1},\FF_{2})=\{X_{1}\wedge X_{2}=0\}
$$ 
where $X_{i}$ is a local holomorphic vector field generating $\FF_{i}$. Note that for arbitrary one dimensional foliations in $M$, the zero set of $X_{1}\wedge X_{2}$ is not a divisor, but since second order differential equations are given by global sections $X_{\FF_{i}}\in H^{0}(M, \DD\otimes T^{*} \FF_{i})$ we have that:
$$
X_{\FF_{1}}\wedge X_{\FF_{2}}\in H^{0}(M, det(\DD)\otimes T^{*} \FF_{1}\otimes T^{*}\FF_{2}).
$$
We have just proved the following proposition.

\begin{proposition}\label{tangencia-entre-dos-ecuaciones-diferenciales}
The tangency divisor between two second order differential equations is given by
$$
tang(\FF_{1}, \FF_{2})=c_{1}(det(\DD)\otimes T^{*}\FF_{1}\otimes T^{*}\FF_{2}).
$$
\end{proposition}
\section{The space of second order differential equations}
We denote by $\EE (a,b)=\Pp H^{0}(M, \DD \otimes \OO_{M}(a,b))$ the space of second order \-differen\-tial equations with bidegree $(a,b)$ in $\PP$. Let us consider the following maps
\begin{eqnarray*}
R_{1}(a,b): \Pp H^{0}(M,\OO_{M}(a,b))\rightarrow \EE(a-2,b+1),& R_{1}(F)=F. X_{\LL},\\ 
R_{2}(a,b): \Pp H^{0}(M,\OO_{M}(a,b))\rightarrow \EE(a+1,b-2),& R_{2}(F)=F. X_{\VV} 
\end{eqnarray*}
where $\LL$ and $\VV$ are the foliations of examples (\ref{equaçao-das-retas}), (\ref{foliacion-vertical}) respectively.\\

Observe that the class of a surface in $M$, $[S]=ah+b\hh$ satisfies: $a, b \geq 0$ (because $a= S.\widetilde{l}$, $b=S.F$). Using this observation and corollary (\ref{el-bigrado-de-una-ecuacion-diferencial}) one can describe the space $\EE(a,b)$ for some bidegrees:

\begin{eqnarray*}
\EE(a-2,b+1)=Im R_{1}(a,b) for \hspace{0.1cm} 0\leq a \leq 2,\hspace{0.1cm} b \geq 0;\\
\EE(a+1,b-2)=Im R_{2}(a,b) for \hspace{0.1cm} 0\leq b \leq 2,\hspace{0.1cm} a \geq 0,
\end{eqnarray*}

In order to find the dimension of this spaces, we have the following proposition.

\begin{proposition}\label{cohomologia-de-O_{M}(a,b)}
For any $a,b \geq 0$, the equalities
$$
h^{0}(\OO_{M}(a,b))=\frac{(a+1)(b+1)(a+b+2)}{2}
$$
and
$$
h^{i}(\OO_{M}(a,b))=0, \hspace{0.1cm} for \hspace{0.1cm} i\geq 1,
$$
hold true.
\end{proposition} 
\begin{proof}
Let us denote by $X=\PP \times \Pd$ and consider $M$ as the incidence variety on $X$. Then clearly we have
\begin{equation}
\xymatrix{
0 \ar[r] &\OO_{X}(-1,-1) \ar[r]& \OO_{X} \ar[r] & \OO_{M} \ar[r] & 0
}
\end{equation}
(here the definition of $\OO_{X}(a,b)$ is the obvious one), and therefore
\[
\xymatrix{
0 \ar[r] &\OO_{X}(a-1,b-1) \ar[r]& \OO_{X}(a,b) \ar[r] & \OO_{M}(a,b) \ar[r] & 0.
}
\]
To conclude is enough to observe that
$$
h^{0}(\OO_{X}(a,b))=\frac{(a+1)(a+2)(b+1)(b+2)}{4}, \hspace{0.1cm} for \hspace{0.1cm} a, b \geq -1
$$
and
$$
h^{i}(\OO_{X}(a,b))=0, \hspace{0.1cm}  for \hspace{0.1cm} i \geq 1 \hspace{0.1cm}  and \hspace{0.1cm} a, b \geq -1.
$$
\end{proof}

For the other cases, one can also calculate the dimension  of the space of di\-ffe\-ren\-tial equations. 

\begin{proposition}\label{dimension-del-espacio-de-ecuaciones-diferenciales}
For $a,b \geq 1$ we have
$$
dim \EE (a,b) = \frac{1}{2}(2a^{2}b+2ab^{2}+3a^{2}+3b^{2}+12ab+9a+9b)-1.
$$
\end{proposition}
\begin{proof}
From sequence (\ref{distribucion-de-contacto-rectas-y-vertical}) one gets
\[
\xymatrix{
0 \ar[r] &\OO_{M}(a+2,b-1) \ar[r]& \DD\otimes \OO_{M}(a,b) \ar[r] & \OO_{M}(a-1,b+2) \ar[r] & 0.
}
\]
The proposition follows from proposition (\ref{cohomologia-de-O_{M}(a,b)}).
\end{proof}
Observe that when $a,b \geq 1$ one has the inclusions
\begin{eqnarray*}
Im R_{1}(a+2,b-1) = A \subseteq \EE(a,b),\\
Im R_{2}(a-1,b+2) = B \subseteq \EE(a,b).
\end{eqnarray*}

After counting dimensions, we conclude that $\EE(a,b)$ is covering by the lines joining points of $A$ and $B$, i.e. $\EE(a,b)=Join(A,B)$.
\begin{proposition}
Every element of $\EE(a,b)$ is a linear combination of an element of $A$ and an element of $B$.
\end{proposition}

\begin{remark}
Since we have an identification between $\Pp(H^{0}(M, \OO_{M}(d,k)))$ and the space of $k$-webs of degree $d$ in $\PP$ (see \cite{Jos}), then in the particular case when $a$ (respectively $b$) is equal to $1$ one has that $B$ (respectively $A$) can be identified with the space of $(b+2)$-webs of degree $0$ (respectively curves of degree $a+2$). 
\end{remark}

Let us consider, for an element $\FF$ in $\EE(a,b)$ which is not in A or B, the following divisors in $M$
\begin{eqnarray*}
tang(\FF, \LL)=(a-1)h+(b+2)\hh\\
tang(\FF, \VV)=(a+2)h+(b-1)\hh.
\end{eqnarray*}

Therefore if $a=1$ we obtain
$$
tang(\FF, \LL)=(b+2)\hh
$$
which corresponds to a $(b+2)$-web $\WW$ of degree $0$  in $\PP$. The lifting of each leaf of $\WW$ is a leaf of $\LL$ (because is a line) and since is in the tangency divisor, it is also a leaf of $\FF$. So $\FF$ has a one-parameter family of lines which are solutions. Observe also that  
\begin{eqnarray*}
A= Im R_{1}(3,b-1)\\
B= Im R_{2}(0,b+2)
\end{eqnarray*}
so, the assertion is true for all the elements of $\EE(1,b)$.\\

Consider now the case $a=2$. In this case
$$
tang(\FF, \LL)=h+(b+2)\hh
$$
corresponds to a $(b+2)$-web $\WW$ of degree 1 in $\PP$. Since a generic $(b+2)$-web of degree 1 has $(b+2)^{2}+(b+2)+1$ invariant lines (see section (\ref{seccion-sobre-jouanolou-para-webs})), and the lifting of this lines are $\LL$-invariant, we would like to say the same for the generic element of $\EE(2,b)$ (note again that the assertion is true for the elements of $A$ and $B$). Then we need the following lemma.

\begin{lemma}
The following maps
\begin{eqnarray*}
T_{1}: H^{0}(M,\DD\otimes \OO_{M}(a,b))\rightarrow H^{0}(M, O_{M}(a-1,b+2)),& T_{1}(X)=X\wedge X_{\LL}\\ 
T_{2}: H^{0}(M,\DD\otimes \OO_{M}(a,b))\rightarrow H^{0}(M, O_{M}(a+2,b-1)),& T_{2}(X)=X\wedge X_{\VV} 
\end{eqnarray*}
are surjective.
\end{lemma}

\begin{proof}
We do the proof for $T_{1}$. Observe first that 
$$
ker(T_{1})=\{F X_{\LL}: F \in H^{0}(M,\OO_{M}(a+2,b-1))\}.
$$
Therefore, using propositions (\ref{cohomologia-de-O_{M}(a,b)}) and (\ref{dimension-del-espacio-de-ecuaciones-diferenciales}) we obtain
$$
dim(Im T_{1})=dim H^{0}(M, O_{M}(a-1,b+2))
$$
and this conclude the proof. 
\end{proof}

Using this lemma one has that the generic element of $\EE(a,b)$ has $(b+2)^{2}+(b+2)+1$ invariant curves which are lifts of lines on $\PP$. We can do the same analysis with the divisor $tang(\FF, \VV)$ in the cases when $b$ is equal to 1 or 2 and we get:

\begin{proposition}
The following assertions hold true when $a,b \geq 1$.
\begin{enumerate}
\item Every element of $\EE(1,b)$ has a one parameter family of invariant curves which are lifts of lines on $\PP$ and the generic element of $\EE(2,b)$ has $(b+2)^{2}+(b+2)+1$ invariant curves which are lifts of lines on $\PP$.
\item Every element of $\EE(a,1)$ has a one parameter family of invariant curves which are lifts of lines on $\Pd$ (that is, fibers of $\pi$) and the generic element of $\EE(a,2)$ has $(a+2)^{2}+(a+2)+1$ invariant curves which are lifts of lines on $\Pd$.
\end{enumerate}
\end{proposition}
For the case when $a \geq 3$ we have our main result.

\begin{theorem}\label{jouanolou-para-ecuaciones-diferenciales}
A generic second order differential equation of bidegree $(a,b)$ with $a \geq 3$ has no invariant algebraic curves which are lifts of curves on $\PP$. Moreover, when $a, b \geq 3$, the generic second order differential equation of bidegree $(a,b)$ does not admit any algebraic solution. 
\end{theorem}
For the proof we need an analogous result for $k$-webs in $\PP$ which we explain in the next section.
\section{About Webs with algebraic leaves on $\PP$}\label{seccion-sobre-jouanolou-para-webs}
\subsection{Global webs on $\PP$}
We are interested in webs defined in the projective plane $\PP$. Let $\WW=[\omega]\in \Pp H^{0}(\PP,\Sym^{k}\Omega^{1}_{\PP} \otimes \mathcal{N})$ be a $k$-web on $\PP$. Analogously to the case of foliations, we define the degree of $\WW$ as the number of tangencies, counted with multiplicities, of $\WW$ with a line not everywhere tangent to $\WW$, and we denote it by $deg(\WW)$. More precisely, if we have $i:\Pp^{1}\hookrightarrow \PP$ a line on $\PP$ then the image of $i$ is everywhere tangent to $\WW$ if and only if $i^{*}\omega$ vanishes identically. When this line is not invariant by $\WW$ the points of tangency with $\WW$ correspond to the zeroes of $[i^{*}\omega]\in \Pp H^{0}(\Pp^{1}, \Sym^{k} \Omega^{1}_{\Pp^{1}}\otimes i^{*}\mathcal{N})$. Observing that $i^{*}\NW = i^{*}\OO_{\PP}(deg(\NW))=\OO_{\Pp^{1}}(deg(\NW))$ and $\Sym^{k} \Omega^{1}_{\Pp^{1}}=\OO_{\Pp^{1}}(-2k)$ we conclude that
$$
deg(\NW)=d+2k.
$$
From now we denote by $\Ww(k,d)=\Pp H^{0}(\PP,\Sym^{k}\Omega^{1}_{\PP}(d+2k))$ the space of $k$-webs of degree $d$ in $\PP$.\\
 
Let $\WW$ be a $k-$web of degree $d$ in $\PP$. For a point $z\in \PP$ we have $k$ points (not necessarily different) $p_{1}(z),\ldots, p_{k}(z) \in \mathbb{P}(T_{p}\PP)$ corresponding to the directions of $\WW$ at this point. In this way we obtain a surface $S_{\WW}\subseteq M$ which is the union of the lifts of the leaves of $\WW$. It is clear that $S_{\WW}$ intersects a generic fiber of $\pi:M\rightarrow \PP$ in $k$ points (counting with multiplicities), so if we write $[S_{\WW}]=ah+b\hh$ we have: $k=[S_{\WW}].h^{2}=b$. On the other hand we know that the class of the lift of a line $l\subseteq \PP$ is $[\check{l}]=\hh^{2}$, and the points of $S_{\WW} \cap \check{l}$ (for $l$ generic) corresponds to points of $\PP$ where $\WW$ and $l$ are tangent, so one conclude:  $d=[S_{\WW}].\hh^{2}=a$ and therefore:
\begin{equation}\label{levantamiento-de-una-k-web-de-grado-d}
[S_{\WW}]=dh+k\hh
\end{equation}

The restriction of $\DD$ to $S_{\WW}$ defines a foliation (over the regular part of $S_{\WW}$) whose leaves are the lifts of the leaves of $\WW$: In fact, the lift of $\WW$ defines a foliation in $S_{\WW}$ which is tangent to $\DD$.\\

\subsection{Invariant algebraic curves}

Let $\WW$ be a $k$-web of degree $d$ in $\PP$ defined by $\omega$ and let $C\subseteq \PP$ an irreducible algebraic curve. As in the case of lines we say that $C$ is $\WW$-invariant if $i^{*}\omega \equiv 0$ where $i$ is the inclusion of the smooth part of $C$ in $\mathbb{C}^{3}$.\\

Observe that we have an isomorphism between $\Pp H^{0}(\PP, \Sym^{k} T\PP \otimes \OO_{\PP}(d-k))$ and $\Ww(k,d)$ giving by the contraction with the volume form of $\PP$. Therefore the web $\WW$ is defined by an element $X_{\WW} \in \Pp H^{0}(\PP, \Sym^{k} T\PP \otimes \OO_{\PP}(d-k))$.\\ 
 
Let us assume that $C$ is given by the irreducible homogenous polynomial $F$ of degree $r$. Then $C$ is $\WW$-invariant if and only if there exist a homogenous polynomial $H$ of degree $d+k(r-1)-r$ such that
\begin{equation}\label{curva-invatiante}
X_{\WW}(F)= FH
\end{equation} 
where $X_{\WW}(F)$ is the application of $dF$ to $X_{\WW}$.  

\begin{remark}
The important fact of equation (\ref{curva-invatiante}) is that it still works for reducible curves, i.e. if the decomposition of the curve is $F=F_{1}^{n_{1}}. \ldots .F_{k}^{n_{k}}$, then the equation (\ref{curva-invatiante}) holds true if and only if each $F_{j}$ define a $\WW$-invariant curve.
\end{remark}

Consider now the following set:
$$
\mathcal{C}(r)=\{\WW \in \Ww (k,d)/ \exists\hspace{0.1cm} curve\hspace{0.1cm} of\hspace{0.1cm} degree\hspace{0.1cm} r \hspace{0.1cm} \WW-invariant\}.
$$
We have the following proposition.

\begin{proposition}\label{webs-with-curve-of-degree-r}
The set $\mathcal{C}(r)$ is an algebraic closed of $\Ww(k,d)$.
\end{proposition}

\begin{proof}
Denote by $S_{r}=H^{0}(\PP, \OO_{\PP}(r))$ and consider 
$$
\mathcal{Z}(r)\subseteq \Ww(k,d)\times \Pp(S_{r})\times \Pp(S_{d+k(r-1)-r})
$$ 
the subset defined by
$$
\mathcal{Z}(r)=\{(\WW, [F], [H])/ X_{\WW}(F) -FH=0\}.
$$
Observe that the natural map 
$$
\pi: \Ww(k,d)\times \Pp(S_{r})\times \Pp(S_{d+k(r-1)-r}) \rightarrow \Ww(k,d)\times \Pp(S_{r})
$$
takes $\mathcal{Z}(r)$ in the closed set $\Sigma(r)\subseteq \Ww(k,d)\times \Pp(S_{r}) $ formed by the pairs $(\WW, [F])$ such that the curve defined by $F$ is invariant by $\WW$. To conclude is enough to observe that $\mathcal{C}(r)$ is the image of $\Sigma(r)$ via the projection 
$$
\Ww(k,d)\times \Pp(S_{r})\rightarrow \Ww(k,d).
$$  
\end{proof}
\subsection{Webs of degree 0 and 1}

Given a projective curve $C\subseteq \PP$ of degree $k$ and a line $l_{0}\in \Pd$ transverse to $C$ there is a germ of $k$-web $\WW_{C}(l_{0})$ on $(\Pd, l_{0})$ defined by the submersions $p_{1}, \ldots, p_{k}: (\Pd, l_{0}) \rightarrow C$ which describe the intersections of $l \in (\Pd, l_{0})$ with $C$. This webs are called algebraic $k$-webs.\\

It is clear from the definition that the fiber of $p_{i}$ through a point $l \in (\Pd, l_{0})$ is contained in the set of lines that contain $p_{i}(l)$. Consequently the fibers of this submersion are contained in lines.\\

When $C$ is a reducible curve with irreducible components $C_{1}, \ldots, C_{m}$ then $\WW_{C}(l_{0})=\WW_{C_{1}}(l_{0})\boxtimes \ldots \boxtimes \WW_{C_{m}}(l_{0})$.\\

If no irreducible component of $C$ is a line then the leaves of $\WW_{C}(l_{0})$ trough $l$ are the hyperplanes passing through it and tangent to $\check{C}$ at some point $p\in \check{C}$.\\

Consider now the incidence variety $M \subseteq \PP \times \Pd$. Onde defines for every curve $C \subseteq \PP$ its dual web $\WW_{C}$ as the one defined by the surface $\pi^{-1}(C)$ seen as a multisection of $\check{\pi}: M \rightarrow \Pd $ (for more details see \cite{PiPer}, section 1.3). It is easy to verify that the germification of this global web at a generic point $l_{0}\in \Pd $ coincides with $\WW_{C}(l_{0})$ defined before. The following proposition can be found in \cite{PiPer}, proposition 1.4.2,  for the $n$ dimensional case. 

\begin{proposition}
If $C\subseteq \PP$ is a projective curve of degree $k$, then $\WW_{C}$ is a k-web of degree zero on $\Pd$. Reciprocally, if $\WW$ is a $k$-web of degree zero on $\Pd$ then there exists $C \subseteq \PP $ a projective curve of degree $k$ such that $\WW = \WW_{C}$. 
\end{proposition}                                                                                                                                                                    
Therefore one has the description of the webs of degree zero as one-parameter families of lines in the plane.\\

Given now a $k$-web $\WW$ of degree 1 on $\PP$, one can consider its lift $S_{\WW}$ to $M$. Observe first that the projection $\check{\pi}|_{S_{\WW}}: S_{\WW} \rightarrow \Pd$ is dominant and has degree one: a generic fiber of $\check{\pi}: M \rightarrow \Pd$ which is the lift of a line in $\PP$ intersects $S_{\WW}$ in the point corresponding to the unique tangency between this line and $\WW$. Then one obtains a foliation $\FF_{\WW}$ on $\Pd$. A similar argument shows that this foliation has degree $k$. Observe that since $\WW$ could be the product of a web of degree zero with a web of degree 1, $\FF_{\WW}$ could have a codimension 1 singular set.\\

Reciprocally, if we begin with a foliation of degree $k$ $\FF$ in $\Pd$, we obtain in the way a $k$-web of degree one $\WW_{\FF}$ in $\PP$. Moreover, as the reader can verify in \cite{PiPer}, theorem 1.4.8, we have $\WW_{\FF_{\WW}}=\WW$, so this correspondence is in fact an isomorphism between $\Ww(k,1)$ and the space of foliations of degree $k$ in $\Pd$. \\

Assume now that we have a foliation $\FF$ in $\Pd$ with a non-degenerate singularity at $l \in \Pd$. Then the fiber of $\check{\pi}$ over $l$ is contained in $S_{\FF}$. Since this fiber corresponds to the lift of the line that $l$ represents we conclude that $l$ is $\WW_{\FF}$-invariant. Using the fact that a generic foliation of degree $k$ has $k^{2}+k+1$ non-degenerate singularities, one obtains the following proposition.

\begin{proposition}
A generic $k$-web of degree one has $k^{2}+k+1$ invariant lines.
\end{proposition}

Using proposition (\ref{webs-with-curve-of-degree-r}) we conclude that every $k$-web of degree 1 has at least one invariant line. Observe that this webs could also have an infinite number of invariant lines, for example in the case of the product of webs of degree one with webs of degree zero.
\subsection{Webs of degree greater than 2}

For webs of higher degrees, we have the following theorem.
\begin{theorem}\label{jouanolou-para-webs}
A generic $k$-web of degree $d$ in $\PP$ does not admit any invariant algebraic curve if $d \geq 2$. 
\end{theorem} 

First we recall some facts for a generic $k$-web $\WW$ of degree $d \geq 2$ on $\PP$ (see \cite{Jos}):
\begin{enumerate}

\item  The surface $S_{\WW} \subseteq M$ associated to $\WW$ is smooth and its class is given by $[S_{\WW}]=dh+k\hh$.

\item Let $\FF_{\WW}$ be the foliation on $S_{\WW}$ given by the restriction of the contact distribution, or by the lifting of the leaves of $\WW$; then the normal bundle of $\FF_{\WW}$ is given by $N\FF_{\WW} = \OO_{S}(h_{r}+\hh_{r})$, where $h_{r}$ and $\hh_{r}$ are the restriction of $h$ and $\hh$ to $S$. 

\item If we write the web in coordinates $(x,y)\in \CC^{2}$ as 
$$
\omega=a_{0}(x,y)dx^{k}+a_{1}(x,y)dx^{k-1}dy+\ldots+a_{k}(x,y)dy^{k}
$$
then $S_{\WW}$ is given by the zero set of 
$$
F(x,y,p,q)=a_{0}(x,y)q^{k}+a_{1}(x,y)q^{k-1}p+\ldots+a_{k}(x,y)p^{k}
$$
and the foliation $\FF_{\WW}$ is defined by the restriction of the vector field  
$$
X=(F_{p}\frac{\partial}{\partial x}-F_{x}\frac{\partial}{\partial p})+(F_{y}\frac{\partial}{\partial q}-F_{q}\frac{\partial}{\partial y})
$$
to $S_{\WW}$. Then the singular set of $\FF_{\WW}$ is given in this coordinates by $\{F=F_{q}=F_{p}=qF_{x}+pF_{y}=0\}$ which is a finite set.
\end{enumerate}
Let us suppose that $\WW$ has an algebraic invariant curve $C$ and let $\widetilde{C}$ its lifting to $M$, which is contained in $S$. 

\begin{lemma}
We have that $\widetilde{C} \cap sing(\FF_{\WW}) \neq \emptyset$.
\end{lemma}

\begin{proof}
Let us suppose that $\widetilde{C} \cap sing(\FF_{\WW}) = \emptyset$, then by Camacho-Sad formula $\widetilde{C}^{2}=0$ and therefore
$$
N\FF_{\WW} . \widetilde{C}=\widetilde{C}^{2} + Z(\FF_{\WW}, \widetilde{C})=0
$$
(see \cite{Brunella}) which is not possible since $h_{r}.\widetilde{C}+ \hh_{r}.\widetilde{C}$ is a positive number.
\end{proof}

Remember that we can identify the set of $k$-webs of degree $d$ with $\Pp H^{0}(M, \OO_{M}(d,k))$. Consider now the algebraic set
$$
\Ss =\{(S_{\WW}, z) \in \Ww(k,d) \times M: z \in sing(\FF_{\WW})\}\subseteq \Ww(k,d) \times M
$$
and its projection over the second factor $\pi_{2}: \Ss \rightarrow M$. Observe that follows from remark (3) that for each $z\in M $ the fiber $\pi_{2}^{-1}(z)$ is a linear subspace of $\Ww(k,d) \times \{z\}$.\\ 

Since for every $z_{1}$, $z_{2} \in M$ there is a biholomorphism $F$ of the form $F(p,[v])=(f(p), Df(p)v)$, for some $f \in Aut(\PP)$, sending $z_{1}$ in $z_{2}$, we conclude that all the fibers of $\pi_{2}$ are smooth, irrecuble and isomorphic, which implies that $\Ss$ is irreducible. Observe also that the other projection $\pi_{1}: \Ss \rightarrow \Ww(k,d)$ is a generically finite map (by Baum-Bott formula is clear that this map is dominant), then $dim \Ss = dim \Ww(k,d)$.\\

Fix now a polynomial $\chi \in \QQ[t]$ of degree one different of the Hilbert polynomial of a fiber of $\pi$ and set $H=Hilb_{\chi}(M)$ the Hilbert scheme of M with respect to $\chi$. If we denote by $H(\DD)$ the subset of $H$ consisting of the subschemes of $M$ tangent to $\DD$ and with Hilbert polynomial $\chi$, then we shall prove in the next section that $H(\DD)$ is a closed subset of $H$. Then we can set the closed set $D\subseteq \Ww(k,d) \times M$ defined as
$$
D=\{(S_{\WW},z) \in \Ww(k,d) \times M: \exists Y \in H(\DD),\hspace{0.1cm} z\in Y,\hspace{0.1cm} Y \subset S_{\WW}\}.
$$ 

Let us assume the theorem for $(k-1)$-webs of degree $d\geq 2$ and suppose that $\pi_{1}(D)=\Ww(k,d)$; that is, every $k$-web of degree $d$ has an algebraic invariant curve whose lifting has Hilbert polynomial $\chi$. By the lemma $\pi_{1}$ sends a dense subset of $D \cap \Ss$ over a dense subset of $\Ww(k,d)$. Therefore $\pi_{1}(D \cap \Ss)=\Ww(k,d)$, but since $\Ss$ is irreducible and has the same dimension that $\Ww(k,d)$ we conclude that $D\cap \Ss= \Ss$.\\

To conclude the theorem we choose a generic $(k-1)$-web of degree $d$ $\WW_{1}$ with no algebraic invariant curves and a pencil of lines through a point $\GG$ such that there exist a sintularity $z$ of $\FF_{\WW_{1}}$ which is not in $S_{\GG}$ and take $\WW=\WW_{1}\boxtimes \GG$. Then we note that $S_{\WW}=S_{\WW_{1}}\cup S_{\GG}$  and through $z$, which is a singularity of $\FF_{\WW}$ we do not have any invariant curve different from a fiber, which is a contradiction. Since there are only countable many Hilbert polynomials, we conclude the theorem.

\subsection{Webs on complex surfaces}
We can obtain a similar result for webs on complex surfaces. Let $S$ be a compact complex surface and we set $\Ww(k,\NN)=\Pp H^{0}(S, \Sym^{k}\Omega^{1}_{S}\otimes \NN)$ the space of $k$-webs on $S$ with normal bundle $\NN$. 

\begin{theorem}
Let $\NN$ be an ample line bundle. Then for $r \gg 0$ the $k$-web on $S$ induced by a very generic element of $\Ww(k, \NN^{\otimes r})$ has not algebraic invariant curves. 
\end{theorem}

\begin{proof}
Following the notation of the previous section, one can define the closed set $D=\{(S_{\WW},z) \in \Ww(k,\NN) \times M: \exists Y \in H(\DD),\hspace{0.1cm} z\in Y,\hspace{0.1cm} Y \subset S_{\WW}\}$, where $M=\Pp(TS)$. Then for $r\gg 0$ we can choose integers $r_{i}\gg 0$ which add up to $r$ and $\omega_{i}\in \Pp H^{0}(S, \Omega^{1}_{S}\otimes \NN^{\otimes r_{i}})$ such that the foliation on $S$ defined by $\omega_{i}$ has not algebraic invariant curves (see \cite{CouPer} theorem 1.1). Then is clear that $\omega=\omega_{1} \ldots \omega_{k}$ is an element of $\Ww(k, \NN^{\otimes r})$ which is not in the image of $D$ by the first projection.
\end{proof}

\begin{remark}
Unlike Theorem \ref{jouanolou-para-webs}, this result is not effective in the sense that we do not describe the webs whose normal bundle is not sufficiently ample.
\end{remark}

\section{Proof of Theorem \ref{jouanolou-para-ecuaciones-diferenciales}}
Let us consider a nonsingular distribution of codimension one $D$ defined by a 1-form $\alpha$ over a complex manifold $M$.
\begin{lemma}\label{subesquema-tangen-a-distribucion}
Let be $Y\subseteq M$ an irreducible subvariety (not necessarily regular) of codimension $k$ and let $\II$ be its ideal sheaf. Then $Y$ is tangent to $D$ if and only if for every $(f_{1}, \ldots , f_{k})\in \II^{\oplus k}$, the equality $$
\alpha \wedge df_{1} \wedge \ldots \wedge df_{k} |_{Y} \equiv 0
$$
holds true.
\end{lemma}
\begin{proof}
Let us suppose that $Y$ is tangent to $D$, that is, $Y$ is tangent to $D$ at the smooth points, and take $f_{1}, \ldots , f_{k} \in \II$. Then, at a regular point $y \in Y$ we can take an open neighborhood $V \subseteq M$ and local coordinates $(x_{1}, \ldots, x_{n})$ such that $Y \cap V=\{x_{1}= \ldots x_{k}=0\}$. Therefore we can write 
$$
f_{i}|_{V}=\sum_{j=1}^{k} a_{ij}x_{j}, \hspace{0.2cm} \alpha=b_{1}dx_{1}+\ldots+ b_{n}dx_{n}
$$
for some analytic functions $a_{ij}$, $b_{i}$. So, by hypotesis we have
$$
b_{k+1}|_{V\cap Y}=\ldots=b_{n}|_{V \cap Y}=0.
$$
Since
\begin{eqnarray*}
\alpha \wedge df_{1} \wedge \ldots \wedge df_{k} |_{V}&=&h \alpha \wedge dx_{1} \wedge \ldots \wedge dx_{k} |_{V}\\
&=& h(b_{k+1}dx_{k+1} \wedge dx_{1}\wedge \ldots \wedge dx_{k}+ \ldots b_{n}dx_{n} \wedge dx_{1}\wedge \ldots \wedge dx_{k} )
\end{eqnarray*}
(here $h$ is an invertible function) we obtain that
$$
Y\cap V \subseteq sing(\alpha \wedge df_{1}\wedge \dots \wedge df_{k})
$$
and this shows the assertion. The other implication is left to the reader.
\end{proof}

Let us return to the case when $M=\Pp(T\PP)$ is the contact variety and $D=\DD$ is the contact distribution. Fix now a polynomial $\chi \in \QQ[t]$ of degree one and set $H=Hilb_{\chi}(M)$ the Hilbert scheme of M with respect to $\chi$. Define also $H(\DD)$, a subset of $H$, by
$$
H(\DD)=\{Y\in H: Y \hspace{0.1cm}is\hspace{0.1cm} tangent \hspace{0.1cm} to\hspace{0.1cm} \DD \}.
$$ 
\begin{remark}
It is easy to show that $\DD$ has no invariant subvarieties of dimension two, so the condition on the degree of $\chi$ is not really necessary.
\end{remark}

\begin{proposition}
$H(\DD)$ is a closed subset of $H$.
\end{proposition}  

\begin{proof}
Let $\UU$ be the universal family in $M \times H$
\[
\xymatrix{
&\UU \ar[dl]_{q_{1}} \ar[dr]^{q_{2}} &\subseteq  M\times H \\
M& &H
}
\]

Remember that for any $Y \subseteq M$ closed subscheme one has the following exact sequence (see \cite{Hrs}, section 2.8)
\[
\xymatrix{
\frac{\II}{\II^{2}} \ar^-{\delta}[r]&\Omega^{1}_{M}|_{Y} \ar[r]&\Omega^{1}_{Y} \ar[r]& 0
} 
\]
where $\delta(f)=df|_{Y}$ and $\II$ is the ideal sheaf of $Y$. This means, writing $\KK_{Y}=Im (\delta)$ (the conormal sheaf of $Y$), that one has
\[
\xymatrix{
0 \ar[r]&\KK_{Y} \ar[r]&\Omega^{1}_{M}|_{Y} \ar[r]&\Omega^{1}_{Y} \ar[r]&0
}.
\]
Then we can consider $\KK$ defined by
\[
\xymatrix{
0 \ar[r]&\KK \ar[r]&(\Omega^{1}_{M\times H |H})|_{\UU} \ar[r]&\Omega^{1}_{\UU|H} \ar[r]&0
}.
\]
By doing the exterior product by the contact form $\alpha$ we have a map
$$
\xymatrix{
\bigwedge^{2}\KK \ar^-{\theta}[r] &(q_{1}^{*}\Omega^{1}_{M}) \otimes \LL
}
$$ 
for some line bundle $\LL$. We conclude by lemma (\ref{subesquema-tangen-a-distribucion}) that
$$
H(\DD)=\{Y \in H : \theta _{Y}=0\}.
$$
Let $\MM$ be a very ample sheaf over $\UU$. Give an integer $r \gg 0 $ it follows by Serre's theorem that there exists a positive integer $N$ and a surjective map $\beta: \OO_{\UU}^{\oplus N} \rightarrow \bigwedge^{2}\KK \otimes \MM ^{\otimes r}$. Denoting by $\sigma$ the composition
\[
\xymatrix{
\OO_{\UU}^{\oplus N}\ar^-{\beta}[r]\ar@/_{7mm}/[rr]_{\sigma} &\bigwedge^{2}\KK \otimes \MM ^{\otimes r}\ar^-{\theta}[r] &(q_{1}^{*}\Omega^{1}_{M}) \otimes \LL \otimes \MM ^{\otimes r}
}
\]
we have, since $\beta$ is surjective, that 
$$
H(\DD)=\{Y \in H: \sigma_{Y}=0\}.
$$ We conclude the proof apllying the next lemma, which is exactly the lemma (2.2) of \cite{CouPer}, to $\mathfrak{X}=\UU$, $T=H$, $p=q_{2}$, $\FF=(q_{1}^{*}\Omega^{1}_{M}) \otimes \LL \otimes \MM ^{\otimes r}$ and $\GG=\OO_{S}^{\oplus N}$, and using theorem 3.8.8 of \cite{Hrs} to obtain $R^{1}p_{*}\FF=0$.  
\end{proof}

\begin{lemma}
Let $p: \mathfrak{X} \rightarrow T$ be a projective morphism. Assume that $\FF$ is a $p$-flat coherent $\OO_{\mathfrak{X}}$-module such that $R^{1}p_{*}\FF=0$. If $\GG$ is a quasi-coherent $\OO_{T}$-modulo and $\sigma: p^{*}\GG \rightarrow \FF$ is a homomorphism of $\OO_{\mathfrak{X}}$-modules, then the set $\{t \in T: \sigma_{t}=0\}$ is closed in $T$.
\end{lemma}

As a consequence one concludes the following proposition.

\begin{proposition}
The set 
$$
\ZZ_{\chi}=\{X \in \EE(a,b): \FF_{X} \hspace{0.1cm} has \hspace{0.1cm} an \hspace{0.1cm} invariant \hspace{0.1cm} subscheme \hspace{0.1cm} of \hspace{0.1cm} Hilbert \hspace{0.1cm} polynomial \hspace{0.1cm}\chi\}
$$
is closed in $\EE(a,b)$.
\end{proposition}

\begin{proof}
Let us denote $\Sigma= \Pp H^{0}(M, TM \otimes \OO_{M}(a,b))$. By the proposition (2.1) of \cite{CouPer} we have that the set
$$
Z=\{(X,Y)\in \Sigma \times H: Y \hspace{0.1cm} is \hspace{0.1cm} \FF_{X}-invariant\}
$$
is a closed set of $\Sigma \times H$. Observe that we have an inclusion $\EE(a,b) \subseteq \Sigma$, so 
$$
Z \cap (\EE(a,b)\times H(\DD)) \subseteq \EE(a,b)\times H
$$
is a closed subset. Now it is enough to note that $\ZZ_{\chi}$ is the natural projection of this closed set.
\end{proof}
We conclude now the proof of the theorem \ref{jouanolou-para-ecuaciones-diferenciales}. Denoting by $\chi_{0}$ to the Hilbert polynomial of a fiber (of $\pi$) $F$, and taking a $(b+2)$-web $\WW$ of degree $a-1$, with $a \geq 3$ and $b \geq 1$, we have that $F.X_{\VV}$ is an element of $\EE(a,b)$, where $F$ is the section corresponding to the surface $S_{\WW}$. By theorem \ref{jouanolou-para-webs} we can choose $\WW$ with no algebraic invariant curves and then $F.X_{\VV}\notin \ZZ_{\chi}$ for every $\chi \neq \chi_{0}$. Since there are only countable many Hilbert polynomials, we conclude the first part of the theorem.\\

When $a,b \geq 3$, we take a $(b-1)$-web $\WW$ of degree $a+2$ and so $F. X_{\LL}$ has bidegee $(a,b)$ (here again $F$ is the section corresponding to $S_{\WW}$). We choose $\WW$ without singular points and with no algebraic invariant curves. Since the curves which are lifting of curves on $\PP$ have different Hilbert polynomial from $\chi_{0}$, $F.X_{\LL}$ is not in $\ZZ_{\chi_{0}}$. This finishes the proof. 
\section{Second order differential equations on complex surfaces}
Let $S$ be any complex compact surface and $M=\Pp(TS)$ the contact variety. We recall that $H^{*}(M)$ is generated as a $H^{*}(S)$-algebra by the chern class of $\OO_{M}(1)$. We denote by $\EE(\NN, k)=\Pp H^{0}(M,\DD \otimes \pi^{*}(\NN)\otimes \OO_{M}(1)^{\otimes k})$ the space of second order differential equations on $S$ with cotangent bundle $\pi^{*}(\NN)\otimes \OO_{M}(1)^{\otimes k}$.\\

\begin{remark}
For any $\WW \in \Ww(k, \NN)$ wich is given locally by $\omega=a_{0}(x,y)+a_{1}(x,y)\frac{dy}{dx}+\ldots+a_{k}(x,y)\frac{dy}{dx}^{k}
$, the associated surface $S_{\WW}$ is given by the zero set of $F(x,y,p)=a_{0}(x,y)+a_{1}(x,y)p+\ldots+a_{k}(x,y)p^{k}$, where $(x,y)$ are local coordinates on $S$. Doing the change of coordinates we see that $[S_{\WW}]=\pi^{*}(\NN)\otimes \OO_{M}(1)^{\otimes k}$.
\end{remark}

Folowing the ideas of the main theorem we can get a same result for any surface:

\begin{theorem}
Let $\NN$ be an ample line bundle on $S$ and $r \gg 0$. Then for any $k \geq 1$ the second order differential equation defined by a very generic element of $\EE(\NN^{\otimes r},k-2)$ has no algebraic solutions. 
\end{theorem}

\begin{proof}
 Using the notation of the previous section one can define the closed set $\ZZ_{\chi}=\{X \in \EE(\NN^{\otimes r},k-2): \FF_{X} \hspace{0.1cm} has \hspace{0.1cm} an \hspace{0.1cm} invariant \hspace{0.1cm} subscheme \hspace{0.1cm} of \hspace{0.1cm} Hilbert \hspace{0.1cm} polynomial \hspace{0.1cm}\chi\}$. To conclude the theorem we take $\WW \in \Ww(k, \NN^{\otimes r} \otimes K_{S}^{*})$ with no algebraic invariant curves ($K_{S}$ is the canonical bundle of $S$) and observe that $F_{\WW}.\VV \in \EE(\NN^{\otimes r},k-2)$, because $T^{*}\VV= \pi ^{*}(K_{S})\otimes \OO_{M}(-1)^{\otimes 2}$ (see the Euler sequence in diagram (\ref{diagrama-del-fibrado-normal-de-la-distribucion-de-contacto})). Clearly  $F_{\WW}.\VV$ is not an element of $\ZZ_{\chi}$.

\end{proof}

\end{document}